\documentclass[twoside,11pt]{article}

\usepackage{graphicx, subfigure}
\usepackage[top=1in,bottom=1in,left=1in,right=1in,footnotesep=0.5in]{geometry}
\usepackage[sort&compress]{natbib} 
\usepackage{amsfonts, amsmath, amssymb, amsthm, constants, bbm}
\usepackage{MnSymbol}
\usepackage{enumitem}
\usepackage{booktabs}
\usepackage{algorithm, algorithmic}

\usepackage{color}
\definecolor{darkred}{RGB}{100,0,0}
\definecolor{darkgreen}{RGB}{0,100,0}
\definecolor{darkblue}{RGB}{0,0,150}

\usepackage{hyperref}
\hypersetup{colorlinks=true, linkcolor=darkred, citecolor=darkgreen, urlcolor=darkblue}
\usepackage{url}

\def\d{{\rm d}}

\newtheorem{theorem}{Theorem}
\newtheorem{proposition}{Proposition}

\newtheorem{contribution}{Contribution}

\theoremstyle{remark}

\newtheorem{remark}{Remark}

\def\beq{\begin{equation}} 
\def\eeq{\end{equation}}
\def\beqn{\begin{eqnarray*}}
\def\eeqn{\end{eqnarray*}}
\def\Bitem{\begin{itemize}\setlength{\itemsep}{.2in}}
\def\bitem{\begin{itemize}\setlength{\itemsep}{.05in}}
\def\eitem{\end{itemize}}
\def\Benum{\begin{enumerate}\setlength{\itemsep}{.2in}}
\def\benum{\begin{enumerate}\setlength{\itemsep}{.05in}}
\def\eenum{\end{enumerate}}
\def\bmult{\begin{multline*}}
\def\emult{\end{multline*}}
\def\bcenter{\begin{center}}
\def\ecenter{\end{center}}
\def\bframe{\begin{frame}}
\def\eframe{\end{frame}}

\newcommand{\thmref}[1]{Theorem~\ref{thm:#1}}

\newcommand{\secref}[1]{Section~\ref{sec:#1}}
\newcommand{\figref}[1]{Figure~\ref{fig:#1}}

\newcommand{\algref}[1]{Algorithm~\ref{alg:#1}}

\DeclareMathOperator*{\argmax}{arg\, max}

\def\cK{\mathcal{K}}
\def\cL{\mathcal{L}}

\def\cS{\mathcal{S}}

\def\bbR{\mathbb{R}}

\def\bbX{\mathbb{X}}

\def\bbZ{\mathbb{Z}}

\newcommand{\E}{\operatorname{\mathbb{E}}}
\renewcommand{\P}{\operatorname{\mathbb{P}}}

\def\Bin{\text{Bin}}

\def\1{\mathbbm{1}}

\definecolor{purple}{rgb}{0.4,.1,.9}

\begin{document}
\thispagestyle{empty}

\title{Estimation of the Global Mode of a Density: \\ Minimaxity, Adaptation, and Computational Complexity}
\author{Ery Arias-Castro\footnote{University of California, San Diego, California, USA (\url{https://math.ucsd.edu/\~eariasca/})} 
\and Wanli Qiao\footnote{George Mason University, Fairfax, Virginia, USA (\url{https://mason.gmu.edu/\~wqiao/})}
\and Lin Zheng\footnote{University of California, San Diego, California, USA (\url{https://math.ucsd.edu/people/graduate-students/})}
}
\date{}
\maketitle

\begin{abstract}
We consider the estimation of the global mode of a density under some decay rate condition around the global mode. We show that the maximum of a histogram, with proper choice of bandwidth, achieves the minimax rate that we establish for the setting that we consider. This is based on knowledge of the decay rate. Addressing the situation where the decay rate is unknown, we propose a multiscale variant consisting in the recursive refinement of a histogram, which is shown to be minimax adaptive. These methods run in linear time, and we prove in an appendix that this is best possible: There is no estimation procedure that runs in sublinear time that achieves the minimax rate. 
\end{abstract}


\section{Introduction} 
\label{sec:intro}

The global mode of a bounded density $f$ on $\bbR^d$ is simply $\argmax_{x \in \bbR^d} f(x)$, which we assume here to be a singleton. It is a particularly important parameter when the density is assumed to be (strongly) unimodal --- in which case it is simply referred to as the mode. In what follows, we say mode to refer to the global mode, even when the density may have multiple local maxima. The problem of estimating the mode of a density dates back to \cite{parzen1962estimation}, who considered a plug-in estimator consisting first in estimating the entire density by kernel density estimation --- a method which had been proposed by \cite{rosenblatt1956remarks} only a few years earlier --- and then in locating the mode of that density estimate. The problem has received a considerable amount of attention since then, partly because it is a prototypical example of a nonparametric point estimation problem -- indeed, one does not need to work through a density estimator to estimate other location parameters such as the mean or median. We refer the reader interested in this long and rich history to a recent survey paper by \cite{chacon2020modal}, where the estimation of multiple local modes is also discussed in light of its intimate connection to the problem of clustering \citep{hartigan1975clustering}. 

Although a number of methods have been proposed in the literature, Parzen's approach and its close variants appears to be the most popular and has been thoroughly studied over the years.
\cite{parzen1962estimation} proved that the estimator was asymptotically normal under some conditions. These conditions were refined over the years, including in a paper by \cite{chernoff1964estimation} who looked at using the uniform kernel (he calls the resulting kernel density estimator ``naive''), which does not satisfy the conditions imposed in  \citep{parzen1962estimation}. The asymptotic normality of the kernel density plug-in estimator was extended to the multivariate setting by \cite{konakov1973asymptotic}, \cite{samanta1973nonparametric}, and later \cite{mokkadem2003law}, who also established various laws of the iterated logarithm.

In this paper we study a closely related method which is based on histograms rather than kernel density estimates. Although the method cannot be said to be really new, it enjoys a number of desirable properties. In the process of establishing and discussing these properties, we also consider questions of convergence rates, computational complexity, and parameter tuning (the bandwidth in Parzen's method).

\subsection{Working assumptions}
\label{sec:setting}
Our basic assumption is that the underlying density behaves like the power function with exponent $\beta$ near its mode and that it is bounded away from its maximum elsewhere. Specifically, we assume that {\em $f$ has a unique mode at $x_0$}, and that, for some $0 < c_0 < C_0$, $h_0 > 0$ and $\beta > 0$,
\begin{equation}
\label{near_mode}
f(x_0) - C_0 \|x-x_0\|^\beta 
\le f(x) 
\le f(x_0) - c_0 \|x-x_0\|^\beta, 
\quad \text{when } \|x-x_0\| \le h_0,
\end{equation}
\begin{equation}
\label{away_mode}
f(x) \le f(x_0) - c_0 h_0^\beta,
\quad \text{when } \|x-x_0\| \ge h_0.
\end{equation}
For convenience, we will also assume that 
\begin{equation}
\label{compact}
\text{\em $f$ has compact support.}
\end{equation}

\subsection{Convergence rates}

Rates of convergence are already implied in \citep{parzen1962estimation}, and were subsequently studied under various assumptions on the underlying density in \citep{eddy1980optimum, vieu1996note, romano1988weak, donoho1991geometrizing, abraham2004asymptotic}, as well as the other publications on the topic mentioned so far. For example, when the density is twice differentiable with bounded second derivative, Parzen's estimator with optimal choice of bandwidth achieves the rate $O(n^{-1/5})$. This was shown to be minimax optimal by \cite{donoho1991geometrizing} under the same conditions as ours displayed in \secref{setting}. Essentially, the density is assumed to have a unique global mode and to behave quadratically in a neighborhood of that mode.
See also the lower bound derived by \cite{tsybakov1990recursive}, although the setting is a little bit different. \cite{romano1988weak} derives local minimax rates with respect to a neighborhood around the density defined by densities that are close up to a certain order: if the underlying density is $C^p$ then the neighborhood consists of densities which are themselves and their derivatives up to order $p-1$ pointwise close in a neighborhood of the mode. Actual minimax rates are derived by \cite{klemela2005adaptive} under similar smoothness assumptions.
As it turns out, assuming that the density is log-concave in addition to twice differentiable does not change the minimax rate of convergence \citep{balabdaoui2009limit}. 

\begin{contribution}
We extend the minimax result of \cite{donoho1991geometrizing} to the general situation where the density behaves like a power function with arbitrary exponent $\beta > 0$ in a neighborhood of its mode. We complement this by showing that the methods we propose achieve the minimax rate. 
\end{contribution}

Confidence intervals or regions for the mode are discussed in a number of publications \citep{romano1988bootstrapping, doss2019inference, dumbgen2008multiscale, genovese2016non, rufibach2010block, eckle2018multiscale, balabdaoui2009limit} under various settings, and they are at least implicit in the papers mentioned earlier discussing the asymptotic normality of the mode since such an asymptotic limit implies an asymptotically valid confidence region (most often an ellipsoid) when the scale parameters are estimated by plug-in.

\subsection{Computational complexity}

The main reason we work with a histogram rather than a kernel density estimator is computational ease: the maximizer of a histogram can be computed in (average) linear time, both in the dimension and the sample size.  

The question of computational complexity has received some attention and has led to variants such as that of \cite{abraham2003simple} who suggest maximizing a kernel density estimator among the sample points, thus avoiding a possibly costly grid search. This might also be a motivation behind some proposals based on nearest neighbors (or spacings in dimension one) as presented in \citep{sager1978estimation, venter1967estimation, dasgupta2014optimal, dalenius1965mode}. Gradient-based estimates such as the mean-shift algorithm of \cite{fukunaga1975estimation} and the closely related procedure proposed by \cite{tsybakov1990recursive} may also have a computational advantage over a grid search approach depending on the refinement of the grid and the number of iterates.

\begin{contribution}
The methods we propose achieve the minimax estimation rate while having linear computational complexity. We show that this is best possible in the sense that no method with sublinear computational complexity can achieve the minimax rate.
\end{contribution}

\subsection{Parameter tuning}

Parzen's method requires a choice of bandwidth. Most of the effort in this direction has been to optimize the accuracy of estimating the density itself and not so much the mode. For example, one can use cross-validation to choose the bandwidth with the intention of minimizing some measure of estimation error for the density --- see \citep{arlot2010survey} and references therein --- and then proceed with Parzen's approach, meaning compute the kernel density estimate with this choice of bandwidth and locate the mode of that estimate. However, in our setting where we impose a condition of the behavior of the density only in a neighborhood of its mode, it is not at all clear that such an estimator would achieve the minimax rate.
It turns out that it does in a setting where the density is assumed twice differentiable everywhere and with strictly negative second derivative at its mode. 
\cite{balabdaoui2009limit} operate under a different global assumption, that the density is log-concave. A maximum likelihood estimator exists under this so-called `shape constraint' alone, and its mode is shown to be minimax optimal under the additional assumption that the density is twice differentiable at its mode.
\cite{klemela2005adaptive} approaches the problem using (and extending) Lepski's method to select the kernel density estimator bandwidth but tailored to the estimation of the mode. The performance rate of the corresponding procedure is established and shown to match the minimax lower bound also derived in the paper for this adaptive setting. This is done assuming that the density is smooth near its mode. 

The problem of selecting a tuning parameter for the estimation of a mode is otherwise addressed via testing for the significance of modes. This is done in a number of papers \citep{duong2008feature, godtliebsen2002significance, genovese2016non, silverman1981using, chacon2013data, rufibach2010block}. This is closely related to the problem of testing for unimodality. We refer the reader to additional references in \citep{eckle2018multiscale} where that connection is made. 

\begin{contribution}
We propose a method --- which is multiscale in nature and performs a sort of bisection search --- that operates in linear time and achieves the non-adaptive minimax rate in our setting.
\end{contribution}

Our approach has antecedents. 
\cite{robertson1974iterative} describes a method that, in dimension one, iteratively focuses on the shortest interval with a certain number of data points, with that number decreasing at a certain rate. The method is shown to be consistent under some mild conditions.
\cite{sager1979iterative} considers a multivariate version based on convex sets. The method is shown to be consistent, and a (suboptimal) rate of convergence is derived in the one-dimensional setting.
\cite{devroye1979recursive} discusses a method based on kernel density estimates at various bandwidth sizes. The method is shown to be consistent but no rate of convergence is provided.

\subsection{Content}
In \secref{mono} we consider the situation where the behavior of the density near its mode is known, meaning that the constants appearing in \secref{setting} are known. We propose a method based on computing a histogram and locating the bin with maximum count, whose performance we establish. We also state a minimax lower bound for this setting, which matches the performance of our method up to a multiplicative constant.
In \secref{multi} we consider the situation where the behavior of the density near its mode is as described in \secref{setting}, but the constants introduced there are unknown. We propose a form of recursive partitioning, which we show achieves the minimax rate established in \secref{mono}, meaning that the method does as well (up to a multiplicative constant) as an optimal method with oracle knowledge of the behavior of the density in the vicinity of its mode.

\subsection{Notation}
Here and elsewhere in the paper, we work with the supnorm, $\|x\| := \max_i |x_i|$ when $x = (x_1, \dots, x_d)$. This is really without loss of generality as we assume the dimension $d$ to be fixed throughout. Indeed, for the problem of estimating a density mode in the nonparametric setting of \eqref{near_mode}-\eqref{away_mode} there is a (standard) curse of dimensionality.

\section{Known behavior near the mode: monoscale approach}
\label{sec:mono}

In this section we assume that we know the parameters describing (in fact, constraining) the behavior of the density, specifically, the constants $c_0, C_0, h_0, \beta$ in \eqref{near_mode} and \eqref{away_mode}. (The density $f$ and its mode $x_0$ remain, of course, unknown.)
This assumption is rather unrealistic in practice, but it is a good place to start, with the question: {\em What would we do, and how well would we do, if we knew this constants?}

With knowledge of these constants, we propose a very simple method, perhaps the simplest one can think of, which effectively amounts to building a histogram and returning the bin with the largest count. The method is obviously very close to a Parzen's method with a kernel proportional to the indicator of one of the bins defining the histogram. The histogram construction is apparently cruder than its smoother kernel density estimate analog, but both methods achieve the same performance rate and the histogram has the advantage of being faster to compute.

\subsection{Method}

The method we have in mind is very simple: It amounts to partitioning the space into bins of equal size and simply returning the location of the bin with the largest count. The bin size needs to be chosen appropriately, based on the behavior of the density near its mode, in order to achieve the minimax rate. The method is compactly described in \algref{mono}. Despite what is hinted at in the literature, the algorithm clearly runs in $O(dn)$ time if we loop over the sample rather than loop over the bins.

\begin{algorithm}[!htpb]
\caption{Mono-scale Mode Hunting}
\label{alg:mono}
\begin{algorithmic}
\STATE {\bf Input:} point set $x_1, \dots, x_n$ in $\bbR^d$ (assumed drawn iid from a density), bin size $h$
\STATE {\bf Output:} a point $\hat x$ (meant to estimate the mode of the underlying density)
\medskip
\STATE {\bf Create} a sparse array of bin counts, where ${\sf BinCount}(k)$ for $k \in \bbZ^d$ stores the number of points in the hypercube $[k h, (k+1) h)$, with all the counts initialized to 0
\STATE {\bf For} $i = 1, \dots, n$, store $k_i = {\sf floor}(x_i/h)$ and update ${\sf BinCount}(k_i) \gets {\sf BinCount}(k_i)+1$  
\STATE {\bf Identify} $\hat k := \argmax_{i = 1, \dots, n}{\sf BinCount}(k_i)$ 
\STATE {\bf Return} $\hat x := \hat k h$
\end{algorithmic}
\end{algorithm}

We quantify the performance of this method by means of the following asymptotic result. (Although a non-asymptotic bound could be stated, we find this version simpler.)

\begin{theorem}
\label{thm:mono_upper}
There is a constant $A > 0$ depending on the constants in \eqref{near_mode} and \eqref{away_mode} such that the mode estimator returned by \algref{mono} is within distance $A h$ of the true location of the mode with probability at least $1 - A \exp(-n h^{d+2\beta}/A)$.
\end{theorem}

\begin{proof}
Since $f$ is assumed to be compactly supported, it is enough to establish the result when the bin size $h$ is small, and in particular we may take it substantially smaller than $h_0$. 
Also, since $1 - A\exp(-1/A) < 0$ for $A$ large enough, it suffices to establish the result when $n h^{d+2\beta} \ge 1$, which we assume henceforth. 
Below $A_1, A_2, \dots$ are constants that do not depend on $n$ or $h$.

Define 
\begin{equation}
p_k := \int_{[kh, (k+1)h)} f(x) \d x, \quad k \in \bbZ^d,
\end{equation}
which is the probability of one draw from $f$ falling in the bin $[k h, (k+1)h)$.
Also, for a set $\cS \subset \bbR^d$, let $N(\cS)$ denote the number of data points in $\cS$, namely, $N(\cS) := \#\{i: X_i \in \cS\}$. For $\cS$ measurable, we have that $N(\cS)$ is binomial with parameters $n$ and $\int_\cS f(x) \d x$. We also let $N_k$ be short for $N([kh, (k+1)h))$, which is the count for bin $k$.

{\em At the mode.}
First, let's consider what happens at the mode. Let $k_0 = {\sf floor}(x_0/h)$ so that $[k_0 h, (k_0+1)h)$ is the bin that contains the mode. Based on \eqref{near_mode} and the fact that $\|x-x_0\| \le h \le h_0$ for all $x$ in that bin, we have
\begin{align}
p_{k_0}
&= \int_{[k_0 h, (k_0+1)h)} f(x) \d x \\
&\ge \int_{[k_0 h, (k_0+1)h)} (f(x_0) - C_0 \|x-x_0\|^\beta) \d x \\
&= f(x_0) h^d - C_0 \int_{[k_0 h, (k_0+1)h)} \|x-x_0\|^\beta \d x \\
&\ge f(x_0) h^d - C_0 \int_{[k_0 h, (k_0+1)h)} \|x-k_0 h\|^\beta \d x \\
&= f(x_0) h^d - C_1 h^{d+\beta}, \quad \text{where } C_1 := C_0 \int_{[0, 1)^d} \|x\|^\beta \d x.
\end{align}
Hence, $N_{k_0}$ is stochastically larger than the binomial distribution with parameters $n$ and $p'_{k_0} := f(x_0) h^d - C_1 h^{d+\beta}$, assuming (as we do) that $h$ is small enough that $p'_{k_0} > 0$.
Applying Bernstein's inequality, we thus establish
\begin{align}
\P\big(N_{k_0} > n p'_{k_0} - s \sqrt{n f(x_0) h^d}\big)
\ge 1 - \exp\big(- \tfrac14 (s^2 \wedge s \sqrt{n f(x_0) h^d})\big),
\end{align}
when $h$ is small enough that $f(x_0) h^d \le 1/2$.
By choosing $s$ such that $s \sqrt{n f(x_0) h^d} = C_1 h^{d+\beta} n$, and taking into account that we assume that $n h^{d+2\beta} \ge 1$, we find that
\begin{align}
N_{k_0} 
&> \tau := n f(x_0) h^d - 2 n C_1 h^{d+\beta}, \label{tau1}
\end{align}
with probability at least $1-\exp(- n h^{d+2\beta}/A_0)$.

{\em Away from the mode.}
We now turn to a bin away from the bin containing the mode. Based on \eqref{near_mode}-\eqref{away_mode}, we have
\begin{align}
p_k
&= \int_{[k h, (k+1)h)} f(x) \d x \\
&\le \int_{[k h, (k+1)h)} \big(f(x_0) - c_0 (\|x-x_0\| \wedge h_0)^\beta\big) \d x \\
&= f(x_0) h^d -c_0 \int_{[k h, (k+1)h)} (\|x-x_0\| \wedge h_0)^\beta \d x \\
&\le f(x_0) h^d -c_0 ((\|k-k_0\|-2) \wedge (h_0/h))^\beta h^{d+\beta},
\end{align}
by the triangle inequality.
For $q \ge 1$ integer, define $\cK_q := \{q \in \bbZ^d : \|k-k_0\| = q+2\}$, and note that 
\begin{equation}
p_k
\le f(x_0) h^d -c_0 (q^\beta \wedge (h_0/h)^\beta) h^{d+\beta}, \quad \forall k \in \cK_q,
\end{equation}
and that $\cK_q$ has cardinality $|\cK_q| \le A_1 q^{d-1}$. We assume henceforth that $h$ is small enough that $c_0 (h_0/h)^\beta \ge 4 C_1$ and restrict ourselves to $q \ge q_0$ where $q_0 \ge 2$ is an integer large enough that $c_0 q_0^\beta \ge 4 C_1$. We now bound the probability that $\hat k$ belongs to some $\cK_q, q \ge q_0$. In view of \eqref{tau1}, we only need to look at the event
\begin{equation}
\max_{q \ge q_0} \max_{k \in \cK_q} N_k > \tau.
\end{equation}
Note that we may restrict our attention to $q \le A_2/h$, since $p_k = 0$ when $kh$ is sufficiently large by the fact that $f$ has compact support. We may assume that $A_2 \ge h_0$.
Therefore, take $q \le A_2/h$ so that $p_k \le f(x_0) h^d -c_0 q^\beta h^{d+\beta}$. 
Then, for $k \in \cK_q$, since $N_k \sim \Bin(n, p_k)$, using Bernstein's inequality, we derive
\begin{align}
\P(N_k > \tau) 
&= \P\big(N_k > n p_k + \tau - n p_k\big) \\
&\le \P\big(N_k - n p_k > n (c_0 q^\beta - 2 C_1) h^{d+\beta}\big) \\
&\le \P\big(N_k - n p_k > n \tfrac12 c_0 q^\beta h^{d+\beta}\big) \\
&\le \exp\left(-\frac{\frac12 (n \tfrac12 c_0 q^\beta h^{d+\beta})^2}{n p_k(1-p_k) + \frac13 (n \tfrac12 c_0 q^\beta h^{d+\beta})}\right) \\
&\le \exp\left(-\frac{\frac12 (n \tfrac12 c_0 q^\beta h^{d+\beta})^2}{n f(x_0) h^d + \frac13 (n \tfrac12 c_0 (h_0/h)^\beta h^{d+\beta})}\right) \\
&\le \exp\big(- q^{2\beta} n h^{d+2\beta}/A_3\big).
\end{align}
Using the union bound, we thus obtain
\begin{align}
\P\big(\max_{q \ge q_0} \max_{k \in \cK_q} N_k > \tau\big)
&\le \sum_{q = q_0}^{{\sf floor}(A_2/h)} A_1 q^{d-1} \exp\big(- q^{2\beta} n h^{d+2\beta}/A_3\big) \\
&\le \sum_{q = q_0}^{\infty} A_1 q^{d-1} \exp\big(- q^{2\beta} n h^{d+2\beta}/A_3\big) \\
&\le \sum_{q = q_0}^{\infty} \exp\big(- q^{2\beta} n h^{d+2\beta}/A_4\big) \\
&\le \int_{q_0-1}^\infty \exp\big(- u^{2\beta} n h^{d+2\beta}/A_4\big) \d u \\
&\le \frac{\exp\big(- (q_0-1)^{2\beta} n h^{d+2\beta}/A_4\big)}{2 \beta (q_0-1)^{2\beta-1} n h^{d+2\beta}/A_4} \\
&\le A_5 \exp\big(- q_0^{2\beta} n h^{d+2\beta}/A_5\big),
\end{align}
using the fact that $q_0 \ge 2$ and $n h^{d+2\beta} \ge 1$ multiple times. The integral was bounded using integration by parts.

We thus have that 
\begin{equation}
N_{k_0} > \tau \ge \max_{q \ge q_0} \max_{k \in \cK_q} N_k
\end{equation}
with probability at least
\begin{equation}
1 - \exp(- n h^{d+2\beta}/A_0) - A_5 \exp\big(- q_0^{2\beta} n h^{d+2\beta}/A_5\big),
\end{equation}
and from this we conclude.
\end{proof}

\subsection{Information bound}

Based on the performance bound established in \thmref{mono_upper}, we can say that \algref{mono} achieves the rate $O(n^{-1/(d+2\beta)})$.
It turns out that this rate is best possible in a minimax sense. This was already known for the exponent $\beta = 2$ \citep{donoho1991geometrizing} --- see also \citep{tsybakov1990recursive}, where the assumed conditions are a little different.
We complete the picture by establishing this as the minimax rate for any value of $\beta > 0$.

\begin{theorem}
\label{thm:mono_lower}
There is a constant $A$ and two densities satisfying the basic properties \eqref{near_mode}-\eqref{away_mode}-\eqref{compact} with modes separated by $n^{-1/(d+2\beta)}/A$ that cannot be distinguished with more accuracy than a probability of error of $1/5$, based on a sample of size $n$.
\end{theorem}

\begin{proof}
The proof is based on Le Cam's two-point prior argument for which a standard reference is \cite[Sec~2.2-2.4]{tsybakov2008introduction}. The idea is to craft two densities, both satisfying the basic properties \eqref{near_mode}-\eqref{away_mode}-\eqref{compact}, that are impossible to distinguish with a high degree of certainty based on a sample of size $n$ and whose modes are on the order of $\asymp n^{-1/(d+2\beta)}$ apart. These densities are denoted $f_1$ and $f_2$ below.

Let $f_1$ be a density on $\bbR^d$, compactly supported, symmetric about the origin (i.e., even as a function), strictly unimodal (and therefore with a unique mode at the origin), and such that $f_1(t) = 1 - \|t\|^\beta$ in a neighborhood of the origin. This implies that there exists $h_0\in(0,1)$ such that $f_1(t) = 1 - \|t\|^\beta$ if $\|t\|\le h_0$, and $f_1(t) \le 1 - h_0^\beta$ if $\|t\| \ge h_0$. Clearly, $f_1$ satisfies the basic properties \eqref{near_mode}, \eqref{away_mode}, and \eqref{compact}.  

Denote the origin by $\underline{0}$ and let $\underline{h}=(h,\cdots,h)\in\bbR^d$ for any $h>0$. We consider $h\le h_0$ below. Define $f_2$ on $\bbR^d$ as follows 
\begin{equation}
f_2(t) = 
\begin{cases}
f_1(t), & \text{if } t\in \bbR^d \setminus (-\underline{h}, \underline{h}) , \\
1-h^\beta, & \text{if } t \in (-\underline{h}, \underline{h}) \setminus (\underline{0}, \underline{h}) , \\
1+(2^d-1)h^\beta-2^{d+\beta}\|t - \underline{h}/2\|^\beta, & \text{if } t \in (\underline{0}, \underline{h}). 
\end{cases}
\end{equation}
See \figref{information_bound} for an illustration.

\begin{figure}[ht]
\centering
\includegraphics[scale=0.23]{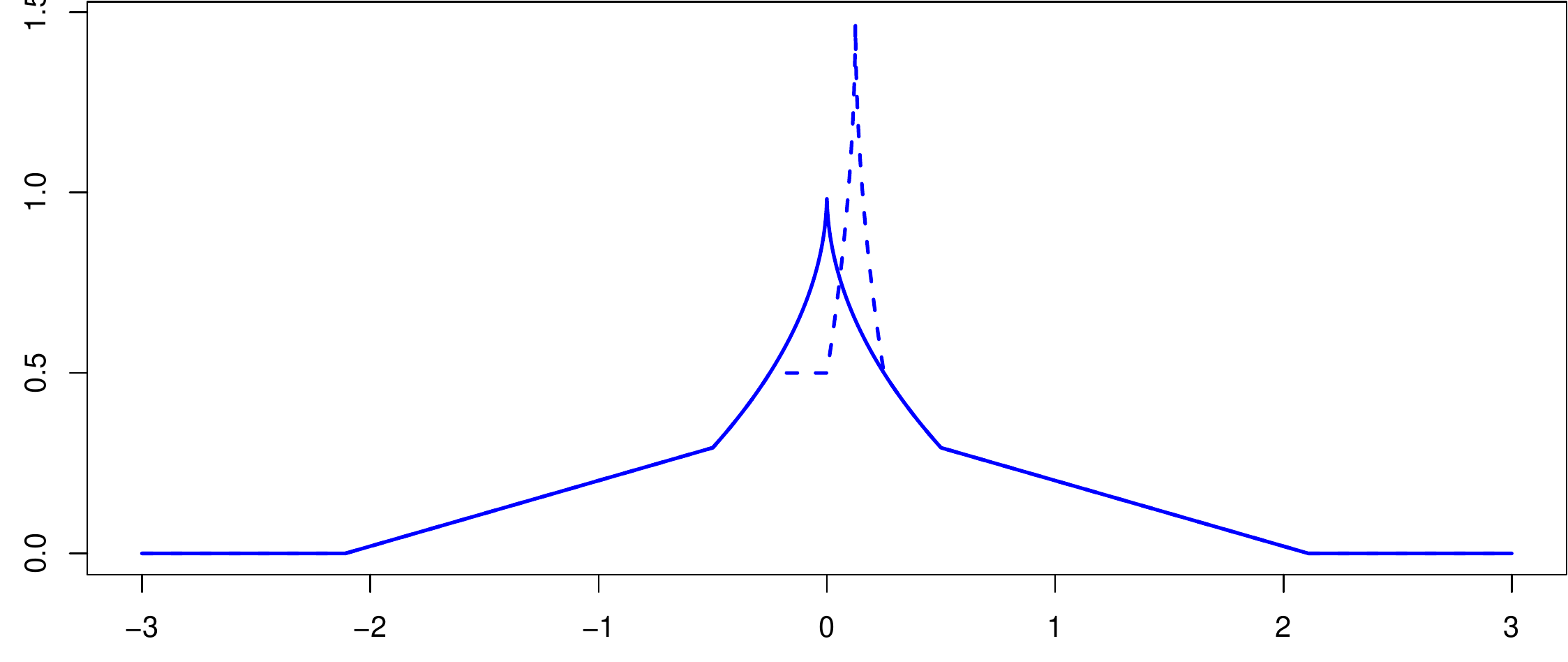}
\includegraphics[scale=0.23]{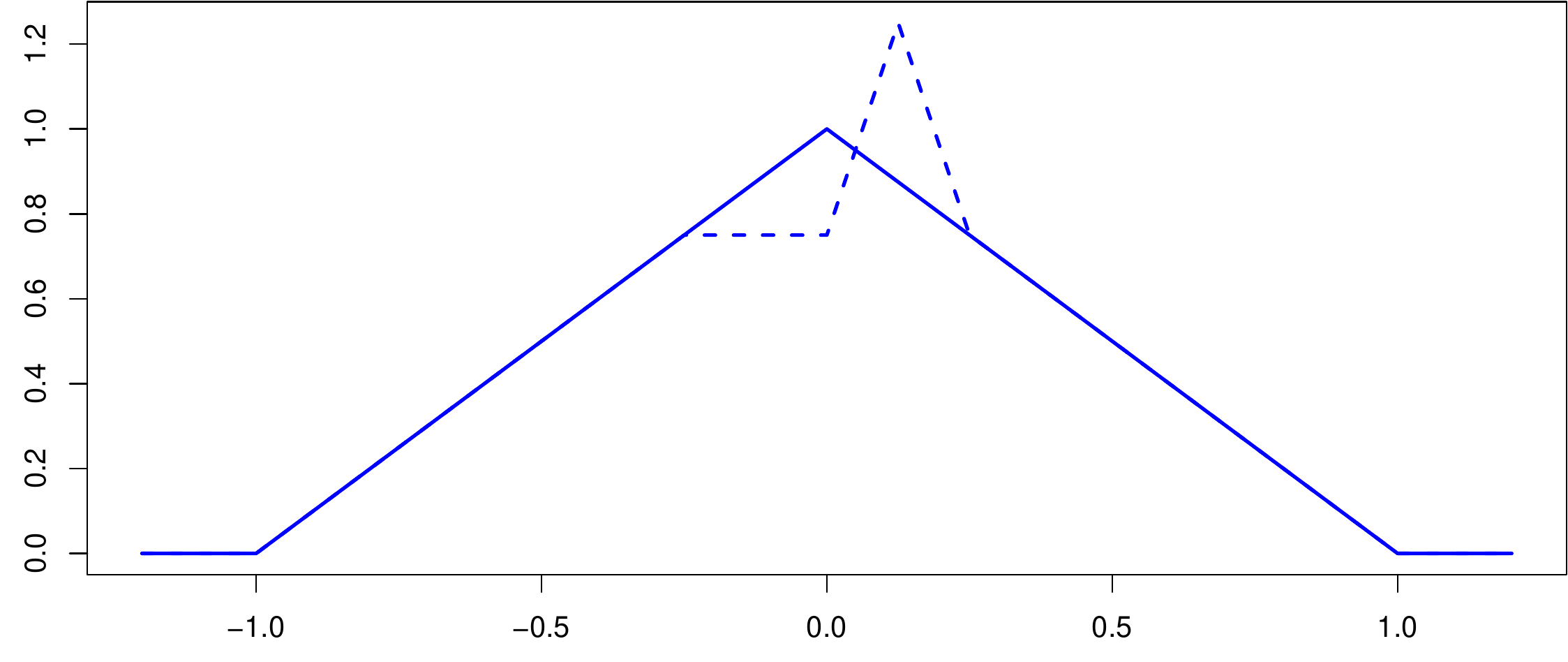}
\includegraphics[scale=0.23]{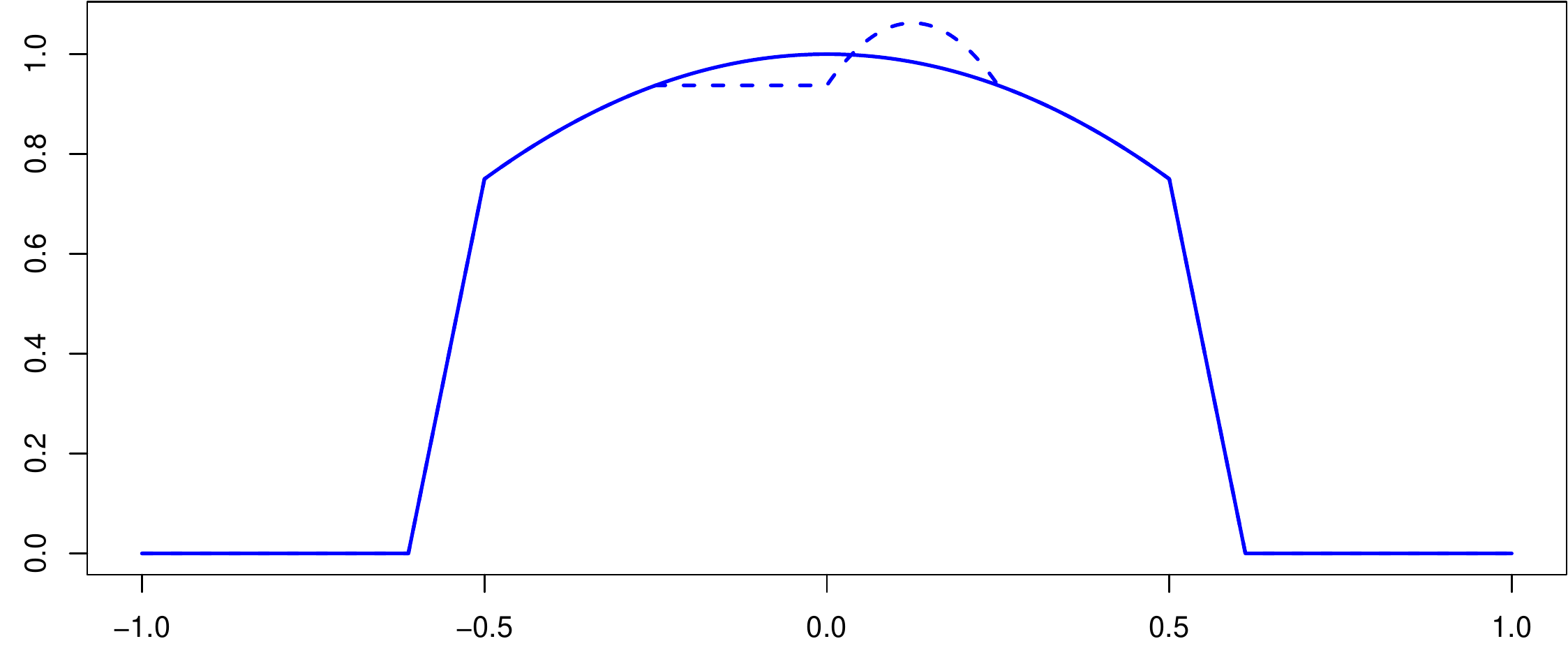}
\caption{\small Examples of pairs of functions $f_1$ (solid) and $f_2$ (dashed) for $\beta = 1/2$, $\beta = 1$, and $\beta = 2$.}
\label{fig:information_bound}
\end{figure}

Notice that we can also write $f_2(t)=f_1(t) + g(t)$, where
\begin{equation}
g(t) = (\|t\|^\beta - h^\beta) \1\{t \in (-\underline{h}, \underline{h}) \setminus (\underline{0}, \underline{h})\} + [(2^d-1)h^\beta-2^{d+\beta}\|t-\underline{h}/2\|^\beta] \1\{t \in (\underline{0}, \underline{h})\}.
\end{equation}
Here $f_2$ is indeed a density function because $f_2\ge 0$ and $\int g(x) dx = 0$ using the fact that
\begin{align}
\int_{(-\underline{h}, \underline{h}) } \|x\|^\beta \d x = \int_0^h s^\beta (2d) (2s)^{d-1}ds = \frac{d}{d+\beta} 2^d h^{d+\beta}.
\end{align}
Observing that $f_2$ has a unique mode at $\underline{h}/2$, below we show that, for all $t$ such that $\|t - \underline{h}/2\| \le h_0 - h/2$, 
\begin{align}\label{f2_bounds}
f_2(\underline{h}/2) - 2^{d+\beta}\|t-\underline{h}/2\|^\beta \le  f_2(t) \le f_2(\underline{h}/2) - 2^{-\beta}\|t-\underline{h}/2\|^\beta.
\end{align}
It is easy to see that \eqref{f2_bounds} holds for $t \in (-\underline{h}, \underline{h})$. For $t\in[\underline{h} - \underline{h_0}, \underline{h_0}] \setminus (-\underline{h}, \underline{h})$, we have
\begin{align}
&f_2(\underline{h}/2) - 2^{d+\beta}\|t-\underline{h}/2\|^\beta\\
&= (1-\|2t - \underline{h}\|^\beta) + (2^d-1) (h^\beta - \|2t - \underline{h}\|^\beta)\\
&\le 1 - \|t\|^\beta = f_2(t),
\end{align}
and
\begin{align}
& f_2(\underline{h}/2) - 2^{-\beta}\|t-\underline{h}/2\|^\beta \\
&= 1+(2^d-1)h^\beta - 2^{-\beta}\|t-\underline{h}/2\|^\beta \\
&\ge 1+(2^d-1)h^\beta - [(h/2)^\beta + \|t\|^\beta] \\
&\ge 1 - \|t\|^\beta = f_2(t).
\end{align}
The last calculation can also be used to show that $f_2(t) \le f_2(\underline{h}/2) - 2^{-\beta}(h_0-h/2)^\beta$ for all $t$ such that $\|t - \underline{h}/2\| \ge h_0 - h/2$. Hence $f_2$ also satisfies the properties \eqref{near_mode}, \eqref{away_mode}, and \eqref{compact}.  

Define $\chi^2(f_2, f_1) := \int f_2^2/f_1 - 1$, which is sometimes called the chi-squared divergence of $f_2$ with respect to $f_1$. According to \cite[Sec~2.2-2.4]{tsybakov2008introduction}, to conclude it suffices to prove that $n \chi^2(f_2, f_1)$ becomes arbitrarily small when $n h^{d+2\beta}$ is small enough. 
We prove this by showing below that $\chi^2(f_2, f_1) = O(h^{d+2\beta})$.
Indeed, elementary calculations yield
\begin{align}
\int \frac{f_2(x)^2}{f_1(x)} \d x - 1
&= \int \frac{g(x)^2}{f_1(x)} \d x\\
&\le \int_{(-\underline{h}, \underline{h})} \frac{(\|t\|^\beta - h^\beta)^2}{f_1(x)}\d x + \int_{(\underline{0}, \underline{h})} \frac{[(2^d-1)h^\beta-2^{d+\beta}\|t-\underline{h}/2\|^\beta]^2}{f_1(x)}\d x \\
& \le \frac{2^d h^{d+2\beta}}{1-h_0^\beta} + \int_{(\underline{0}, \underline{h})} \frac{2^{2(d+\beta)}[h^\beta + \|t-\underline{h}/2\|^\beta]^2}{1-h_0^\beta}\d x\\
& \le \frac{2^{2(d+\beta+3)} }{1-h_0^\beta} h^{d+2\beta},
\end{align}
and from this we conclude.
\end{proof}

\thmref{mono_upper} and \thmref{mono_lower}, together, establish $n^{-1/(d+2\beta)}$ as the minimax rate for estimating the mode under the conditions \eqref{near_mode}-\eqref{away_mode}-\eqref{compact} --- where the emphasis should be on the first one. This extends the result of \cite{donoho1991geometrizing}, who proved this for $\beta = 2$ is dimension $d = 1$. (The method they studied and showed to be minimax was none other than Parzen's method with a proper choice of bandwidth.) It turns out that this is the same rate as under the more restrictive assumption that the density is twice differentiable with bounded second derivative in the vicinity of the mode and with negative definite Hessian at the mode. This was established by \cite{tsybakov1990recursive}, who went further: If the density is H\"older-$\alpha$ with $\alpha \ge 2$ near its mode, and the Hessian there is negative definite, then the minimax rate is $n^{-(\alpha-1)/(d+2\alpha)}$, and is achieved by a gradient ascent method proposed in the same paper. This rate is faster than what it is in our setting where, under the same assumptions, still corresponds to $\beta = 2$.

\paragraph{Computational complexity}
We have therefore established that \algref{mono}, which runs in linear time, is minimax rate optimal when its tuning parameter (the bin size $h$) is properly chosen. Is it possible, however, to do even better in the sense of designing an algorithm that runs in sublinear time that also 	achieves the minimax estimation rate? The answer is `No', and this is general: In a very broad sense, it is not possible to achieve a minimax rate in sublinear time in an estimation problem where, as is the case here, that rate is a negative power of the sample size (perhaps with some poly-logarithmic multiplicative factor). 
See \secref{sublinear} for details.


\section{Unknown behavior near the mode: multiscale approach}
\label{sec:multi}

Choosing the bin size correctly in \algref{mono} is very important. It is completely analogous to choosing the bandwidth to build a histogram or to perform kernel density estimation. All the methods we are aware of necessitate the tuning of parameters whose optimal value, as in our case, depends on the behavior of the density near its mode. 

In the special situation where the density is twice differentiable everywhere and has a negative definite Hessian at the mode --- which necessarily forces $\beta = 2$ --- Parzen's estimate with bandwidth chosen by cross-validation appears to achieve the minimax rate because {\em (i)} the optimal choice of bandwidth is the same, in order of magnitude, for the problem of density estimation and the problem of mode estimation; and {\em (ii)} a choice of bandwidth based on cross-validation achieves the optimal rate for the problem of density estimation as established by \cite{hall1983large} and \cite{stone1984asymptotically}. 

We propose a multiscale method that is able, under some conditions, to zoom in on the mode without assuming much of the underlying density and achieve the minimax error rate. Moreover, the method still operates in (expected) linear time. The method, in principle, still depends on a couple of parameters, but these can be chosen with much less knowledge of the underlying density. 
{\em And by letting these parameters diverge to infinity arbitrarily slowly, the method is, in effect, parameter-free.}

We are not aware of any method that is able to choose these tuning parameters automatically while achieving the minimax performance rate, except for the that of \cite{klemela2005adaptive}. In that paper, the general approach advocated by Lepski for selecting tuning parameters is implemented and shown to yield a choice of bandwidth which leads to an {\em adaptive} minimax rate. Indeed, the paper also derives minimax rates for when the density smoothness at the mode is unknown, and these rates are different from those when the smoothness at the mode is known: there is a price to pay. Under the looser conditions that we operate under, it turns out that there is no price to pay.

\subsection{Method}

When the parameters in \eqref{near_mode}-\eqref{away_mode}, in particular the exponent $\beta$, are unknown, we adopt a recursive partitioning approach. 
The resulting method is described in \algref{multi} where the bin counts are implicitly computed as in \algref{mono}.

\begin{algorithm}[!htpb]
\caption{Multi-scale Mode Hunting}
\label{alg:multi}
\begin{algorithmic}
\STATE {\bf Input:} point set $x_1, \dots, x_n$ in $\bbR^d$, scale multiplier $b \ge 2$, margin $\kappa \ge 0$
\STATE {\bf Output:} a point $\hat x$ (meant to estimate the mode of the underlying density)
\medskip
\STATE Define the finest scale $s_{\rm max} := {\sf floor}(\log(n)/d \log b)$
\STATE Initialize the active set to be $I_{\rm active} \gets \{1, \dots, n\}$
\STATE {\bf For} $s = 1, \dots, s_{\rm max}$
\STATE {\bf For} $k \in \bbZ^d$ identify $I(k, s) = \{i \in I_{\rm active} : x_i \in [k b^{-s}, (k+1) b^{-s})\}$
\STATE {\bf EndFor}
\STATE Identify $\hat k(s) := \argmax_k \# I(k, s)$
\STATE Update $I_{\rm active} \gets \bigcup \{I(k,s) : \|k-\hat k(s)\| \le \kappa\}$
\STATE {\bf EndFor}
\STATE {\bf Return} $\hat x := \hat k(s_{\rm max}) b^{-s_{\rm max}}$
\end{algorithmic}
\end{algorithm}

\begin{proposition}
\algref{multi} runs in linear expected time.
\end{proposition}

\begin{proof}
At each scale, the main computational task is to identify the bins where the sample points that remain active reside. We saw when discussing the computational complexity of \algref{mono} that doing this can be done in time $O(d m)$ if there are $m$ active points.
At scale $s = 1$, all points are active, and so the resulting time complexity is $O(dn)$.
At scale $s > 1$, we expect at most $n \big(f(x_0) [(2\kappa+1) b^{-s+1}]^d \wedge 1\big) \asymp n \big((\bar\kappa b^{-s})^d \wedge 1\big)$ active points to process, where $\bar\kappa := \kappa \vee 1$, resulting in a complexity of $O(d n (\bar\kappa^d b^{-ds} \wedge 1))$ at that scale.
Summing these expected computational costs over $s = 1, \dots, s_{\rm max}$ yields an overall computational cost bounded by a constant multiple of 
\[
d n + \sum_{s \ge 1} d n (\bar\kappa^d b^{-ds} \wedge 1) \asymp d n \big(\tfrac{\log \bar\kappa}{\log b} + 1\big).
\]
(We have assumed that $d$ is constant, as there is a real curse of dimensionality in the context that interests us here, but we carried it throughout these computations to display its influence, which can be seen to be rather benign.)
\end{proof}

\begin{theorem}
\label{thm:multi_upper}
There is a constant $A > 0$ depending on the constants in \eqref{near_mode} and \eqref{away_mode} such that the mode estimator returned by \algref{mono} is within distance $t n^{-1/(d+2\beta)}$ of the true location of the mode with probability at least $1 - (A/\log b) (\kappa b^2/t)^{d+2\beta} \exp(-(t/\kappa b^2)^{d+2\beta}/A)$ whenever $t \ge 1$, $b \ge A$, and $\kappa \ge A$, as well as $(\kappa+1) b^2/t  \le n^{2\beta/d(d+2\beta)}$.
\end{theorem}

The statement is a bit complicated but the core message is simple: if $t$, $b$, and $\kappa$ are understood as remaining constant while the sample size becomes large, the estimator is within distance $t n^{-1/(d+2\beta)}$ with probability at least $1 - A' \exp(- t^{d+2\beta}/A')$ when $t \ge 1$, where this time $A'$ also depends on $b$ and $\kappa$.

\begin{proof}
First, by a simple modification of the arguments underlying \thmref{mono_upper}, there is a constant $A_0$ which depends on the constants \eqref{near_mode} and \eqref{away_mode} such that at scale $s$, whenever $n (b^{-s})^{d+2\beta} \ge 1$ and $b \ge A_0$ as well as $\kappa \ge A_0$, $\|\hat k(s) - k_0(s)\| \le A_0$ with probability at least $1- A_0 \exp(-n (b^{-s})^{d+2\beta}/A_0)$ where $k_0(s) := {\sf floor}(x_0 b^s)$.  This comes from considering $b^{-s}$ as playing the role of $h$ and $\kappa$ as playing the role of $q$ in the proof of \thmref{mono_upper}, and realizing that restricting the density to $[(\hat k(s-1)-\kappa) h, (\hat k(s-1)+\kappa) h)$ does not have any substantial effect. 
In what follows, we assume that $b$ and $\kappa$ are indeed sufficiently large that $b \ge A_0$ and $\kappa \ge A_0$.

Second, for $\hat x$ to be within distance $\delta$ of $x_0$ it suffices that, $\|\hat k(s) - k_0(s)\| \le \kappa$ for some $s$ satisfying $(\kappa+1) b^{-s} \le \delta$.  This is simply because, by design, 
\[
\|\hat x - \hat k(s) b^{-s}\| \le \kappa b^{-s}, \quad \forall s = 1, \dots, s_{\rm max},
\]
and by definition $\|x_0 - k_0(s) b^{-s}\| \le b^{-s}$.
Therefore, $\|\hat x - x_0\| \le \delta$ when $\|\hat k(s) - k_0(s)\| \le \kappa$ for $s = 1, \dots, \bar s(\delta) := {\sf ceiling}(\log_b((\kappa+1)/\delta))$, where $\log_b(x) := \log(x)/\log(b)$ is the logarithm in base $b$.

With these preliminaries, we proceed by bounding the probability that $\|\hat k(s) - k_0(s)\| \le \kappa$ for $s = 1, \dots, \bar s(\delta)$ with $\delta$ chosen as $\delta := t n^{-1/(d+2\beta)}$ where $t > 0$. 
Note that $\bar s(\delta) \le s_{\rm max}$ as our assumptions include $(\kappa+1) b^2/t \le n^{2\beta/d(d+2\beta)}$.
Using the union bound, this probability is
\begin{align*}
&\ge 1 - \sum_{s = 1}^{\bar s(\delta)} A_0 \exp(-n (b^{-s})^{d+2\beta}/A_0) \\
&\ge 1 - \int_1^{\bar s(\delta)+1} A_0 \exp(-n (b^{-s})^{d+2\beta}/A_0) \d s \\
&= 1 - \int_{n b^{-(d+2\beta)(\bar s(\delta)+1)}/A_0}^{n b^{-(d+2\beta)}/A_0} A_0 \exp(-u) \frac1{u (d+2\beta) \log b} \d u \\
&\ge 1 - \frac{A_0}{(d+2\beta) \log b} \frac{\exp(- n b^{-(d+2\beta)(\bar s(\delta)+1)}/A_0)}{n b^{-(d+2\beta)(\bar s(\delta)+1)}/A_0}
\end{align*}
Some elementary calculations give that 
\[
n b^{-(d+2\beta)(\bar s(\delta)+1)}
\ge (t/(\kappa+1) b^2)^{d+2\beta},
\]
and from this we conclude.
\end{proof}

We have thus proved that \algref{multi} achieves the minimax rate without knowledge of the exponent $\beta$ driving the behavior of the density in the vicinity of its mode as prescribed in \eqref{near_mode}. The algorithm can thus be said to be `adaptive' in that sense. 
This is in contrast with the more structured situation that \cite{tsybakov1990recursive} considered. In that situation, \cite{klemela2005adaptive} showed that there is a cost to adaptation, although a small one: a poly-logarithmic factor; a Lepski-type method was shown to be adaptive minimax.




\bibliographystyle{chicago}
\bibliography{ref}

\appendix

\section{Minimax rates in sublinear time}
\label{sec:sublinear}

We establish here that, under rather general conditions, {\em achieving a minimax estimation rate in sublinear time is impossible when the minimax rate converges to zero faster than some negative power of the sample size} --- a situation that is quite general indeed, although there are exceptions such as deconvolution problems \citep{fan1991optimal}.
The fundamental idea is quite straightforward and is based on the fact that a sublinear-time algorithm is not even able to `look' at all observations and thus effectively operates as if on a sample of size sublinear in the available sample size, so that its performance is under the purview of the minimax rate corresponding to that smaller sample size. The remainder of this section is simply devoted to formalizing this discussion.

\begin{remark}
We focus here on a minimax rate based on the sample size and not other parameters of the problem such as the dimension. Since there is a real curse of dimensionality for the problem of estimating of a density mode, we have assumed the dimension to be fixed throughout, but we do believe that no algorithm which is sublinear in the dimension can achieve the minimax rate. 
\end{remark}

Consider a general statistical problem where we have a family of distributions $\{P_\theta: \theta \in \Theta\}$ on some measurable space $\bbX$. The dataset consists in a sample, $X_1, \dots, X_n$, drawn iid from a distribution in that family, say $P_{\theta_0}$. Given some dissimilarity measure on the space, namely $\cL: \bbX \times \bbX \to \bbR_+$ measurable, which plays the role of loss function, the risk of an estimator $\hat\theta = \varphi(X_1, \dots, X_n)$ at $\theta \in \Theta$ is defined as 
\begin{equation}
{\sf risk}_n(\varphi, \theta) := \E_\theta\big[\cL(\varphi(X_1, \dots, X_n), \theta)\big],
\end{equation}
and its worst-case risk is the supremum of that over the entire parameter space,
\begin{equation}
{\sf risk}_n(\varphi) := \sup_{\theta \in \Theta} {\sf risk}_n(\varphi, \theta). 
\end{equation}
Note that $\varphi$ is a (measurable) function on finite sequences of elements of $\bbX$ and $\E_\theta$ above is the expectation with respect to $X_1, \dots, X_n$ iid from $P_\theta$. The minimax risk for this estimation problem is simply the infimum of this quantity over all estimators,
\begin{equation}
R(n) := \inf_\varphi {\sf risk}_n(\varphi).
\end{equation}
We assume throughout that $R(n) < \infty$, at least for $n$ large enough, for otherwise the setting is trivial.
In that case, $R$ is non-increasing as a real-valued function on the positive integers. 

\begin{theorem}
Consider a setting as described above where $\limsup_{a \to \infty} \limsup_{n \to \infty} R(a n)/R(n) = 0$. Then ${\sf risk}_n(\varphi) \gg R(n)$ for any estimator $\varphi$ which can be computed in $o(n)$ time when applied to a sample of size $n$.
\end{theorem}

We assume below that it takes a unit of time to simply register an observation for further processing. 

\begin{proof}
Let $\varphi$ be such an estimator and let $b(n)$ denote the time it takes to compute $\varphi$ on a sample of size $n$ so that $b(n) = o(n)$ by assumption. (We assume that $\varphi$ is not randomized in what follows, but similar arguments apply when in the situation where it relies on an exogenous source of  randomness.) Assuming, without loss of generality, that $\varphi$ registers the first observation first, $\varphi$ is computed as follows: $\varphi(x_1, \dots, x_n) = \psi_k(x_{i_1}, \dots, x_{i_k})$ where $i_1$ is constant equal to 1, $i_2$ is a function of $x_1$, etc, and $i_k$ is a function of $x_{i_1}, x_{i_2}, \dots, x_{i_{k-1}}$, and $\psi_k$ is a function of $k$ variables. 
The number of entries taken in, $k$, need not be constant, but ignoring some variables as needed, we may take $k$ to be constant, and we then let $\psi$ denote $\psi_k$. And given our assumption on the computational complexity of $\varphi$, necessarily, $k \le b(n)$. 
When applied to an iid sample, by independence, $I_2 := i_2(X_1)$ is independent of $X_2, \dots, X_n$ and so $I_2$ is effectively uniformly distributed on $\{2, \dots, n\}$; given $I_2 = i_2$, $I_3 := i_3(X_{i_1}, X_{i_2})$ (remember $i_1 = 1$) is independent of $\{X_i : i \notin \{i_1, i_2\}\}$ and therefore uniform in $[n] \setminus \{i_1, i_2\}$; etc; and given $I_2 = i_2, \dots, I_{k-1} = i_{k-1}$, $I_k := i_k(X_{i_1} \dots, X_{i_{k-1}})$ is independent of $\{X_i : i \notin \{i_1, \dots, i_{k-1}\}\}$ and therefore uniform in $[n] \setminus \{i_1, \dots, i_{k-1}\}$. We may therefore conclude that $\varphi(X_1, \dots, X_n)$ has the same law of $\psi(X_1, \dots, X_k)$. 

Having established this, we then have
\begin{align}
{\sf risk}_n(\varphi, \theta) 
&= \E_\theta\big[\cL(\varphi(X_1, \dots, X_n), \theta)\big] \\
&= \E_\theta\big[\cL(\psi(X_1, \dots, X_k), \theta)\big] \\
&= {\sf risk}_k(\psi, \theta), \quad \forall \theta, \\
\end{align}
implying that
\begin{equation}
{\sf risk}_n(\varphi)
\ge {\sf risk}_k(\psi)
\ge R(k)
\ge R(b(n)).
\end{equation}
We then conclude with the fact that $R(b(n))/R(n) \to \infty$ due to the fact that $b(n)/n \to 0$ and our assumption on $R$.
\end{proof}

\end{document}